\documentclass[11pt,reqno,oneside]{amsart}
 \usepackage{amsmath,amssymb,cite,epsfig,xspace,alltt,verbatim,bm}

\numberwithin{equation}{section}

\theoremstyle{plain}
\newtheorem{theorem}{Theorem}[section]
\newtheorem{lemma}{Lemma}[section]

\theoremstyle{definition}

\theoremstyle{remark}
\newtheorem{remark}{Remark}[section]

\newcommand{\ud}{\, \mathrm{d}}
\newcommand{\Real}{\mathbb R}

\newcommand{\Int}{\mathbb Z}

\newcommand{\bb}{{\mathbf b}}

\newcommand{\Db}{D}
\newcommand{\eb}{{\mathbf e}}

\newcommand{\fb}{{\mathbf f}}
\newcommand{\gb}{{\mathbf g}}

\newcommand{\rhob}{{\bm \rho}}

\newcommand{\sigb}{{\bm \sigma}}
\newcommand{\ub}{{\mathbf u}}
\newcommand{\vb}{{\mathbf v}}
\newcommand{\xb}{{\mathbf x}}
\newcommand{\yb}{{\mathbf y}}
\newcommand{\At}{\widetilde{A}}
\newcommand{\Bt}{\widetilde{B}}
\newcommand{\bbt}{\widetilde{\bb}}
\newcommand{\xbt}{\widetilde{\xb}}

\newcommand{\E}{{\mathcal E}}
\newcommand{\F}{{\mathcal F}}

\newcommand{\lpnorm}[2]{\left\|#1\right\|_{\ell^{#2}_h}}
\newcommand{\Lpnorm}[2]{\left\|#1\right\|_{L^{#2}}}

\newcommand{\half}{{\textstyle \frac{1}{2}}}

\DeclareMathOperator*{\esssup}{ess \ sup}

\begin{document}

\title[An Optimal Order Error Analysis of the Quasicontinuum Approximation]{An Optimal Order Error Analysis of the One-Dimensional Quasicontinuum Approximation}
\author{Matthew Dobson}
\author{Mitchell Luskin}

\address{Matthew Dobson\\
School of Mathematics \\
University of Minnesota \\
206 Church Street SE \\
Minneapolis, MN 55455 \\
U.S.A.}
\email{dobson@math.umn.edu}
\address{Mitchell Luskin \\
School of Mathematics \\
University of Minnesota \\
206 Church Street SE \\
Minneapolis, MN 55455 \\
U.S.A.}
\email{luskin@umn.edu}

\thanks{
This work was supported in part by DMS-0757355,
 DMS-0811039,  the Institute for Mathematics and
Its Applications,
 the University of Minnesota Supercomputing Institute, and
the University of Minnesota Doctoral Dissertation Fellowship.
  This work is also based on
work supported by the Department of Energy under Award Number
DE-FG02-05ER25706.
}



\date{\today}

\begin{abstract}
We derive a model problem for quasicontinuum approximations
that allows a simple, yet insightful, analysis of
the optimal-order convergence rate in the continuum limit
for both the energy-based quasicontinuum approximation and the
quasi-nonlocal quasicontinuum approximation.  The optimal-order
error estimates for the quasi-nonlocal quasicontinuum approximation
are given for all strains up to the continuum limit strain for fracture.
The analysis is based on an explicit treatment of the
coupling error at the atomistic to continuum interface, combined
with an analysis of the error due to atomistic and continuum
schemes using the stability of the quasicontinuum approximation.
\end{abstract}

\maketitle
{
\thispagestyle{empty}

\section{Introduction}
The quasicontinuum method (QC) denotes a class of related
approximations of fully atomistic models for crystalline solids that
reduces the degrees of freedom necessary to compute a deformation to a
desired accuracy~\cite{pinglin03, pinglin05, legollqc05, ortnersuli,
OrtnerSueli:2006d, e05, tadm96, knaportiz, e06, mill02, rodney_gf,
miller_indent, curtin_miller_coupling, jacobsen04,dobs08}.  The
derivation of a quasicontinuum method first removes atomistic degrees
of freedom by using a piecewise linear approximation of the atom
deformations with respect to a possibly much smaller number of
representative atoms.  Since atoms interact significantly with atoms
beyond their nearest neighbors, a further approximation is required to
obtain a computationally feasible method.

In this paper, we will analyze two QC variants that approximate
the total atomistic energy by using a continuum approximation in a
portion of the material called the continuum region.  The deformation
gradient is assumed to be slowly varying in the continuum region,
making the continuum approximation accurate.  The more computationally
intensive atomistic model is used for the remainder of the
computational domain, which is called the atomistic region.  In this
region, all of the atoms are representative atoms, so that there is no
restriction of the types of deformations in the atomistic region.  To
maintain accuracy, the atomistic region must contain all regions of
highly varying deformation, such as material defects.  Adaptive
methods that determine what portion of the domain should be assigned
to the atomistic region in order to acheive the required accuracy have
been considered in~\cite{ortnersuli, OrtnerSueli:2006d, prud06,
oden06, arndtluskin07a, arndtluskin07b, ArndtLuskin:2007c}.  Other
approaches to atomistic to continuum coupling have been developed and
analyzed in~\cite{BadiaParksBochevGunzburgerLehoucq:2007,
ParksBochevLehoucq:2007}, for example.

In Section~\ref{sec:model}, we derive a model problem for QC
approximations and describe the energy-based quasicontinuum (QCE)
approximation and the quasi-nonlocal quasicontinuum (QNL)
approximation.  These two approximations use the same continuum
approximation, but differ in how they couple the atomistic and
continuum regions.  We also give stability estimates for the two QC
variants.

We have derived our model quasicontinuum energy from a general
quasicontinuum energy by expanding each interaction to second order to
be able to present a simple, but illuminating, analysis.  The model
differs from a standard quadratic approximation by keeping certain
first-order terms.  These are a source of leading-order coupling
error, and reflect the behavior in the non-linear case.   We have also
chosen to analyze the model problem for boundary conditions given by
restricting to periodic displacements to maintain the simplicity of
the analysis.

The goal of this paper is to give an error analysis as the number of
atoms per interval increases (the continuum limit).  The residual at
atoms in the coupling interface is lower order (order O($1/h$) and
O(1) for QCE and QNL, respectively) than the residual in either the
atomistic or continuum region (O($h^2$) in all cases).  However, the
corresponding error depends primarily on the {\bf{sum}} of the
residual at the atoms in the atomistic to continuum coupling
interface, and this sum has the higher order O($h$) due to the
cancellation of the lowest order terms when the residual is summed
across the interface.

In Section~\ref{sec:comp}, we split the residual for the QCE
approximation into the part due to the continuum approximation and the
part due to  coupling the atomistic and continuum regions.  The
stability of the QCE approximation and the O($h^2$) estimate for the
corresponding residual combine to give an optimal order bound on the
error due to the continuum approximation.  We then derive an explicit
representation of the coupling error, and we observe that this error
is small and decays away from the interface since the coupling
residual is oscillatory, as described in the preceding paragraph.  The
coupling error is the leading order term in the total error,
dominating the continuum approximation error.

We show that the displacement converges at the rate O($h$) in the
discrete $l^\infty$ norm and the rate O($h^{1/p}$) in the $w^{1,p}$
norms where $h$ is the interatomic spacing.  Our analysis extends the
results of E, Ming, and Yang~\cite{emingyang} that show that the error
is O($1$) in the $w^{1,\infty}$ norm for the QCE method applied to a
problem with harmonic interactions and Dirichlet boundary conditions.

In Section~\ref{sec:qnl}, we present the same analysis in the QNL
case.  Here we show that the improved order of accuracy in the
coupling interface serves to nearly balance the order of the error due
to the continuum approximation, and we are consequently able to give
higher order optimal error estimates for the QNL approximation than
for the QCE approximation.  We
show that the displacement
now converges at the rate O($h^2$) in the discrete $l^\infty$ norm and
the rate O($h^{1+1/p}$) in the
$w^{1,p}$ norms where $h$ is the interatomic spacing.
E, Ming, and Yang~\cite{emingyang} have obtained
O($h$) esimtates in the $w^{1,\infty}$ norm for the Lennard-Jones potential and
 strains bounded away from the continuum limit strain for fracture.
We have obtained optimal order error estimates for the discrete $l^\infty$
and
$w^{1,p}$ norms
for all strains up to the continuum limit strain for fracture.

This paper extends our analysis of the effect of atomistic to
continuum model coupling on the total error in the energy-based
quasicontinuum approximation~\cite{dobs08c} to include external
forcing.  We also now include an analysis of the quasi-nonlocal
approximation.

\section{One-Dimensional, Linear Quasicontinuum Approximation}
\label{sec:model}
We consider the periodic displacement
from a one-dimensional reference lattice with spacing $h=1/N,$  and we denote the
positions of the atoms in the reference lattice by
\begin{equation*}
x_j := jh,\qquad -\infty<j<\infty.
\end{equation*}
We will derive and analyze the linearization about a uniform
deformation gradient $F$ given by the
deformation
\begin{equation}\label{F}
y^F_j := jFh=ja,\qquad -\infty<j<\infty,
\end{equation}
which is a lattice with spacing
$a:=Fh.$
We will then consider perturbations $u_j$ of the
lattice $y^F_j$ which are $2N$ periodic in $j,$ that is, we will
consider deformations $y_j$ where
\begin{equation*}
y_j := y^F_j+u_j,\qquad -\infty<j<\infty,
\end{equation*}
for
\begin{equation}\label{per}
u_{j+2N} = u_j ,\qquad -\infty<j<\infty.
\end{equation}
We will often describe the perturbations $u_j$
satisfying~\eqref{per} as displacements
(which they are if $y^F_j$ is considered the reference lattice).
We thus have that the deformation satisfies
\begin{equation}\label{perF}
y_{j+2N} = y_j+2F,\qquad -\infty<j<\infty.
\end{equation}
We note that neither the
reference lattice spacing $h$ nor the uniform lattice spacing $a$ need be
the equilibrium lattice constant or the well of the interatomic potential.

\subsection{Notation}
Before introducing the models, we fix the following notation.
We define the backward differentiation operator, $\Db \ub,$  on periodic
displacements by
\begin{equation*}
(\Db \ub)_j:=\frac{u_{j}-u_{j-1}}h\qquad\text{for }-\infty<j<\infty.
\end{equation*}
Then $(\Db \ub)_j$ is also $2N$ periodic in $j.$
We will use the short-hand $(\Db \ub)_j=Du_j.$

For periodic displacements, $\ub,$ we define the discrete norms
\begin{equation*}
\begin{split}
\lpnorm{\ub}{p} &:= \left(h \sum^{N}_{j=-N+1}  |u_j|^p\right)^{1/p},
	\qquad 1\le p<\infty,\\
\lpnorm{\ub}{\infty} &:= \max_{-N+1\leq j\leq N}  |u_j|,
\end{split}
\end{equation*}
 By including the
whole period, we ensure that they are all norms (in particular,
$\lpnorm{\ub}{p}=0$
implies $\ub=0$).
We sum over a single period to make the norms finite.
We will also consider periodic functions $u(x):\Real\to\Real$
satisfying
\begin{equation}\label{percon}
u(x+2)=u(x)\qquad\text{for }x\in\Real.
\end{equation}
We define corresponding continuous norms
\begin{equation*}
\begin{split}
\Lpnorm{u}{p} &:= \left( \int^{1}_{-1} |u(x)|^p \ud x \right)^{1/p},
	\qquad 1\le p<\infty,\\
\Lpnorm{u}{\infty} &:= \esssup_{x \in (-1,1)} |u(x)|.\\
\end{split}
\end{equation*}
We let $u'$ denote the weak derivative of the periodic function $u.$
We note that if $\Lpnorm{u'}{p}<\infty,$ then
 $u(x)$ is continuous for all $x$ in $\Real$ and
 $u(-1)=u(1).$  We will similarly denote higher order weak derivatives
 of the periodic function $u$ as
 $u'',$ $u''',$ and $u^{(4)}.$

\subsection{Atomistic Model}
We first consider the total energy per period
\begin{equation}\label{en1}
\E^{tot,h}(\yb) := \E^{a,h}(\yb)-\F(\yb),
\end{equation}
for deformations $\yb$ satisfying \eqref{perF} where the total atomistic energy per period is
\begin{equation}\label{atoma}
\E^{a,h}(\yb)=
\sum_{j=-N+1}^{N}
h \left[\phi\left( \frac{y_{j} - y_{j-1}}{h} \right)
+ \phi\left(\frac{y_{j} - y_{j-2}}{h}\right)\right]
\end{equation}
for a two-body interatomic
potential $\phi$ (assumptions on the potential are given in
Section~\ref{assumptions}),
and where the total external potential energy per period is
\begin{equation}\label{exta}
\F(\yb)= \sum_{j=-N+1}^{N} h f_j y_{j}
\end{equation}
for periodic dead loads $\fb$ such that $f_{j+2N}=f_j$ and
$\sum_{j=-N+1}^{N} f_j = 0.$

We have scaled the atomistic energy per bond in~\eqref{atoma}
by $h\phi(r/h)$ and the external force per atom by $hf_i$ in \eqref{exta}.
This scaling permits a continuum limit as $h \rightarrow 0.$
If $y,f \in C^\infty(\Real)$ satisfy  $y(x+2)=y(x)+2F,$ $y'(x) > 0,$
$f(x+2)=f(x),$
and $\int_{-1}^{1}f(x)\,dx=0;$ if $\phi(r)$ is locally Lipschitz for
$r \in (0,\infty);$ and if we set
$y_j = y(x_j)$ and $f_j = f(x_j),$ then the energy per
period~\eqref{en1} converges to ~\cite{blanc02}
\begin{equation}\label{conta}
\int^1_{-1} \left[\hat\phi\left(y'(x)\right) - f(x) y(x) \right]\ud x
\end{equation}
as $N \rightarrow \infty \ (\text{which implies } h \rightarrow 0),$
where $\hat\phi(r)=\phi(r)+\phi(2r).$
In the following, we linearize the atomistic model which leads to a
corresponding linearized continuum model.  This paper analyzes the
convergence of two quasicontinuum approximations to the minimizer
of the linearized continuum model's total energy.

\subsection{Linearized Atomistic Model}
We will henceforth consider the linearized version of the above energies
while reusing the notation $\E^{a,h}$ and $\E^{tot,h}.$
The total atomistic energy~\eqref{atoma} becomes~\cite{dobs08c}
\begin{equation}\label{at}
\begin{split}
\E^{a,h}(\ub) &:=  \sum_{j=-N+1}^{N}
h \left[ \phi'_F \left[\frac{u_{j} - u_{j-1}}{h}\right]
  + \half \phi''_{F} \left[\frac{u_{j} - u_{j-1}}{h}\right]^2 \right. \\
&\qquad\qquad \left. + \phi'_{2F} \left[\frac{u_{j} - u_{j-2}}{h}\right]
  + \half \phi''_{2 F} \left[\frac{u_{j} - u_{j-2}}{h}\right]^2
  \right],
\end{split}
\end{equation}
for displacements,
$\ub,$ satisfying the periodic boundary
conditions~\eqref{per}.
Here
$\phi'_F := \phi'(F),   \phi''_{F} := \phi''(F),
\phi'_{2F} := \phi'(2F),  \phi''_{2 F} := \phi''(2F),$
where $\phi$ is the interatomic potential in~\eqref{atoma}.
We have removed the additive constant $2\phi(F) + 2\phi(2F)$ from the
quadratic expansion of the energy, and we will remove the additive
constant $-h \sum_{j=-N+1}^N f_j y^F_j$ from $\F(\ub)$ when computing
the external potential of the displacement $\ub.$
We note that the first order terms in \eqref{at} sum to zero by the periodic boundary
conditions and thus do not contribute to the total energy or the equilibrium equations.
We keep the first order terms in the model \eqref{at} since they do
not sum to zero when the atomistic model is coupled to the continuum
approximation in
the quasicontinuum energy.
The atomistic energy~\eqref{at} has the equilibrium equations
\begin{equation}\label{atom}
\begin{split}
(L^{a,h} \ub)_j
 &= \frac{- \phi''_{2 F} u_{j+2} - \phi''_{F} u_{j+1}
+ 2 (\phi''_{F} + \phi''_{2 F}) u_{j} - \phi''_{F} u_{j-1}
- \phi''_{2 F} u_{j-2}}{h^2} =  f_j, \\
&\hspace{1.9in} u_{j+2N} = u_j,
\end{split}
\end{equation}
for $-\infty<j<\infty.$

\subsection{Linearized Continuum Model}
For periodic $u \in C^\infty(\Real)$ and $u_j = u(x_j),$ the total linearized
atomistic energy
\begin{equation}\label{en2}
\E^{tot,h}(\ub) := \E^{a,h}(\ub)-\F(\ub)
\end{equation}
converges to
\begin{equation}\label{cont}
\int^1_{-1} \left[W(u'(x)) - f(x) u(x) \right]\ud x
\end{equation}
as $N \rightarrow \infty,$
where the continuum strain energy density, $W(\epsilon),$ is given by
\begin{equation}
\label{contsplit1}
\begin{split}
W(\epsilon) :=  \left[ (\phi'_F + 2 \phi'_{2F})
\epsilon
+ \half (\phi''_F + 4 \phi''_{2F})
     \epsilon^2 \right].
\end{split}
\end{equation}
 The equilibrium equations~\eqref{atom} are
a five-point consistent difference approximation of the equilibrium equation
of the continuum model~\eqref{cont}, which are
\begin{equation}
\label{bvp}
\begin{split}
- (\phi''_{F}+4\phi''_{2 F}) &u''_e = f,\\
u_e(x+2)&=u_e(x),
\end{split}
\end{equation}
for $x\in\Real.$

Quasicontinuum approximations couple an approximation
of the continuum model with the atomistic model.
The \emph{continuum approximation} consists of a finite element discretization
of the continuum model's elastic
energy.  The discretization uses a continuous, piecewise linear displacement
$\ub$ with the atom positions $\xb$ as nodes.
The external force term is applied as a point force
at each node, so that~\eqref{cont} becomes
\begin{equation}\label{cont_appx}
\sum_{l=-N+1}^{N} h [ W(Du_l) - f_l u_l].
\end{equation}
The continuum approximation has equilibrium equations
\begin{equation*}
\begin{split}
&-(\phi''_{F} + 4\phi''_{2 F})\frac{ u_{l+1}
- 2 u_{l} + u_{l-1}}{h^2} =  f_l, \qquad -\infty < l < \infty,\\
&\hspace{1.in} u_{l+2N} = u_l,
\end{split}
\end{equation*}
which is a three-point consistent difference approximation of the
equilibrium equations for the continuum model~\eqref{bvp}.
In one dimension, the above is actually the standard finite difference
approximation of~\eqref{cont}; however, it is framed in finite
element terminology for flexibility in coarsening, adapativity, and
higher dimensional modelling.

\subsection{Assumptions}
\label{assumptions}
We assume that
\begin{equation}\label{posdef}
\phi''_{F} + 4 \phi''_{2 F}>0,
\end{equation}
which implies that the total linearized atomistic energy~\eqref{at} is positive
definite (up to uniform translation of the displacement).
Thus both equations~\eqref{atom} and~\eqref{bvp}
have a unique solution (up to
uniform translation) provided that
\begin{equation}\label{resultant}
\sum_{j=-N+1}^{N} f_j = 0.
\end{equation}
For simplicity, we assume in the following that
$f$ is odd in addition to being periodic, that is,
\begin{equation}\label{odd}
f(x) = -f(-x)\quad\text{and}\quad f(x+2)=f(x) \qquad \text{for }-\infty<x <\infty,
\end{equation}
which implies that $f_j:=f(x_j)$ satisfies
\begin{equation}\label{odd2}
f_j = -f_{-j}\quad\text{and}\quad f_{j+2N}=f_j \qquad \text{for }\qquad -\infty<j <\infty.
\end{equation}
We obtain a unique, odd periodic solution satisfying the mean value condition
\begin{equation}
\label{meanzero}
\sum_{j=-N+1}^{N} u_j = 0.
\end{equation}
To give nonoscillatory solutions to the
equilibrium equations~\eqref{atom} (that is, to guarantee that the roots of the
corresponding characteristic equation are real~\eqref{roots}),  we further assume that
\begin{equation}
\label{assume}
\phi''_F > 0 \text{ and } \phi''_{2F} < 0.
\end{equation}
The assumption~\eqref{assume} holds for potentials
that allow an accurate second neighbor cut-off, such as the Lennard-Jones potential~\cite{dobs08,dobs08c}.

\subsection{Energy-Based Quasicontinuum Approximation}

The energy-based quasicontinuum approximation (QCE) of $\E^{a,h}(\ub)$
decomposes the reference lattice into an atomistic
region and a coarse-grained continuum region.  It computes
a total energy by using the atomistic energy~\eqref{at} in the atomistic
region and by using the continuum approximation~\eqref{cont_appx}
to sum the energy of the continuum region.

For our analysis, we will consider an atomistic region defined by
the atoms with reference positions $x_j$ for $j=-K,\dots,K,$
and a continuum region containing the remaining atoms, $j=-N+1,\dots,-K-1$ and
$j=K+1,\dots,N.$
All atoms in the continuum region, along with the two atoms on the boundary,
$j = \pm K$ will act as nodes for the continuum approximation.
 The continuum region can be decomposed into elements
$(x_{(l-1)},x_l)$ for $l=-N+1,\dots -K$ and $l=K+1,\dots, N.$
(In general, elements can contain many atoms of the reference lattice,
but in this paper we do not consider coarsening in the continuum region.)

To construct the contribution of the atomistic region to
the total quasicontinuum energy, it is convenient to construct
an energy associated with each atom by splitting equally the energy
of each bond to obtain
\begin{equation}
\label{atomsplit}
\begin{split}
\E^{a,h}_j(\ub) :=
\frac{h}{2}  \Bigg[ \phi'_F &\left[\frac{u_{j+1} - u_{j}}{h}\right]
  + \half \phi''_{F} \left[\frac{u_{j+1} - u_{j}}{h}\right]^2 \\
&+ \phi'_{2F} \left[\frac{u_{j+2} - u_{j}}{h}\right]
  + \half \phi''_{2 F} \left[\frac{u_{j+2} - u_{j}}{h}\right]^2 \Bigg]\\
+\frac{h}{2}  \Bigg[ \phi'_F &\left[\frac{u_{j} - u_{j-1}}{h}\right]
  + \half \phi''_{F} \left[\frac{u_{j} - u_{j-1}}{h}\right]^2 \\
&+ \phi'_{2F} \left[\frac{u_{j} - u_{j-2}}{h}\right]
  + \half \phi''_{2 F} \left[\frac{u_{j} - u_{j-2}}{h}\right]^2 \Bigg].
\end{split}
\end{equation}
The continuum energy~\eqref{cont_appx} is split into energy per element
$h W(Du_l)$ where $W$ is given in~\eqref{contsplit1}, and
$h=x_l-x_{l-1}$ is the length of the continuum element
$(x_{l-1},x_l).$

To construct a quasicontinuum approximation QCE that conserves exactly
the energy of atomistic model~\eqref{at} for lattices $y^F_j$ given
by a uniform deformation gradient $F$ (see~\eqref{F})
the elements
$(x_{-K-1},\,x_{-K})$ and $(x_K,\,x_{K+1})$ on the border
of the atomistic region
should contribute only one half of the continuum energy associated with
that element.  The QCE energy is then
\begin{equation}
\label{qceTot}
\begin{split}
\E^{qce,h}(\ub) :=
 &\sum_{l = -N+1}^{-K-1} h W(Du_l)
+ \half h W(Du_{-K})
+ \sum_{j=-K}^{K}   \E_{j}^{a,h}\left(\ub\right) \\
&+ \half h W(Du_{K+1})
+ \sum_{l= K+2}^{N} h W(Du_l).
\end{split}
\end{equation}

The equilibrium equations for the total QCE energy, $\E^{qce,h}(\ub)-\F(\ub),$ then take
the form~\cite{dobs08,dobs08c}
\begin{equation}\label{qc1}
L^{qce,h} \ub_{qce}-\gb = \fb,
\end{equation}
where, for $0 \leq j \leq N,$ we have
\begin{equation*}
\label{Lqce}
\begin{split}
(L^{qce,h} \ub)_j &= \phi''_F \frac{-u_{j+1} +2 u_j - u_{j-1}}{h^2} \\
&+ \begin{cases}
\displaystyle
 4 \phi''_{2 F} \frac{-u_{j+2} +2 u_j - u_{j-2}}{4 h^2},
& 0 \leq j \leq K-2, \\[6pt]
\displaystyle
 4 \phi''_{2 F} \frac{-u_{j+2} +2 u_j - u_{j-2}}{4 h^2}
+ \frac{\phi''_{2 F}}{h} \frac{u_{j+2} - u_{j}}{2 h},  & j = K-1, \\[6pt]
\displaystyle
 4 \phi''_{2 F} \frac{-u_{j+2} +2 u_j - u_{j-2}}{4 h^2}
- \frac{2 \phi''_{2 F}}{h} \frac{u_{j+1} - u_{j}}{h}
+ \frac{\phi''_{2 F}}{h} \frac{u_{j+2} - u_{j}}{2 h},
& j = K, \\[6pt]
\displaystyle
4 \phi''_{2 F} \frac{-u_{j+1} +2 u_j - u_{j-1}}{h^2}
-  \frac{2 \phi''_{2 F}}{h} \frac{u_{j} - u_{j-1}}{h}
+ \frac{\phi''_{2 F}}{h} \frac{u_{j} - u_{j-2}}{2 h}, & j = K+1, \\[6pt]
\displaystyle
4 \phi''_{2 F} \frac{-u_{j+1} +2 u_j - u_{j-1}}{h^2}
+ \frac{\phi''_{2 F}}{h} \frac{u_{j} - u_{j-2}}{2 h}, & j = K+2, \\[6pt]
\displaystyle
4 \phi''_{2 F} \frac{-u_{j+1} +2 u_j - u_{j-1}}{h^2}, & K+3 \leq j \leq N,
\end{cases}
\end{split}
\end{equation*}
with $\gb$ given by
\begin{equation}
\label{ghost}
g_j = \begin{cases}
0, & 0 \leq j \leq K-2, \\
-\frac{1}{2h} \phi'_{2F}, & j = K-1, \\
\hphantom{-} \frac{1}{2h} \phi'_{2F}, & j = K, \\
\hphantom{-} \frac{1}{2h} \phi'_{2F}, & j = K+1, \\
-\frac{1}{2h} \phi'_{2F}, & j = K+2, \\
0, &  K+3 \leq j \leq N.\\
\end{cases}
\end{equation}
For space reasons, we only list the entries for $0\le j\le N.$  The equations
for all other $j\in\Int$ follow from symmetry and periodicity.
Due to the symmetry in the definition of the atomistic and continuum regions,
we have that $L^{qce,h}_{i,j} =  L^{qce,h}_{-i,-j}$  and
$g_{j} = -g_{-j}$ for $-N+1 \leq i,j \leq 0.$  To see this, we define
the involution operator $(S\ub)_j=-u_{-j}$ and observe that
$\E^{qce,h}(S\ub)=\E^{qce,h}(\ub).$
It then follows from the chain rule that
\[
S^TL^{qce,h} S\ub-S^T\gb - S^T\fb=L^{qce,h} \ub-\gb - \fb\quad \text{for all periodic }\ub.
\]
Since $S^T=S$ and the assumption~\eqref{odd2} is equivalent to $Sf=f,$ we can
conclude that
\begin{equation}\label{S}
SL^{qce,h} S = L^{qce,h}\quad\text{and}\quad Sg=g.
\end{equation}
Furthermore, we can conclude that the
 unique mean zero
solution~\eqref{meanzero} to
the equilibrium equations~\eqref{qc1} is odd.  This follows from $S^{-1}=S$ and
\eqref{S} which together imply that $S\ub$ is a solution
if and only if $\ub$ is.  Because $S$ preserves
the mean zero property, we conclude that $\ub_{qce}$ is odd.

\subsection{Stability of the Quasicontinuum Operator}
Our analysis of the QCE error will utilize the following stability results
for the operator $L^{qce,h}.$
\begin{lemma}
\label{qce_stab}
If $\nu:=\phi''_{F} - 5 |\phi''_{2 F}|>0,$ then
\begin{equation}\label{stab}
h\vb\cdot L^{qce,h}\vb\ge \nu \lpnorm{\Db \vb}{2}^2.
\end{equation}
\end{lemma}
\begin{proof}
The stability result~\eqref{stab} follows from the identity
\begin{equation*}
\begin{split}
\half h \vb\cdot L^{qce,h}\vb=
 & \sum_{l = -N+1}^{-K-1} h \widehat{W}(Dv_l)
+ \half  h \widehat{W}(Dv_{-K})
+  \sum_{j=-K}^{K}   \widehat{\E}_{j}^{a,h}\left(\vb\right) \\
&+ \half h \widehat{W}(Dv_{K+1})
+ \sum_{l= K+2}^{N} h \widehat{W}(Dv_l),
\end{split}
\end{equation*}
where
\begin{equation}
\label{atomsplit2}
\begin{split}
\hat\E^{a,h}_j(\vb) &:=
\frac{h}{2} \Bigg[ \half \phi''_{F} \left[\frac{v_{j+1} - v_{j}}{h}\right]^2
  + \half \phi''_{2 F} \left[\frac{v_{j+2} - v_{j}}{h}\right]^2 \Bigg]\\
&\qquad+\frac{h}{2} \Bigg[
   \half \phi''_{F} \left[\frac{v_{j} - v_{j-1}}{h}\right]^2
  + \half \phi''_{2 F} \left[\frac{v_{j} - v_{j-2}}{h}\right]^2 \Bigg]
\end{split}
\end{equation}
and
\begin{equation}
\label{contsplit2}
\widehat{W}(\epsilon) :=
 \half (\phi''_F + 4 \phi''_{2F}) \epsilon^2.
 \end{equation}

We then have that
\begin{align*}
h\vb\cdot L^{qce,h}\vb &\geq h \sum_{j=-N+1}^N \half \phi''_F \left[(Dv_{j+1})^2
+ (Dv_{j})^2\right]
- h \sum_{j = -N+1}^{-K-1}  4|\phi''_{2F}|\left(Dv_{j}\right)^2 \\
&\qquad - 2 h |\phi''_{2F}| (Dv_{-K})^2
 - h \sum_{j=-K}^{K}   |\phi''_{2F}| \left[(Dv_{j+2})^2+(Dv_{j+1})^2
                                       +(Dv_{j})^2+(Dv_{j-1})^2\right] \\
&\qquad - 2 h |\phi''_{2F}| (Dv_{K+1})^2 - h \sum_{j= K+2}^{N} 4|\phi''_{2F}|
                             \left(Dv_{j}\right)^2 \\
&\geq (\phi''_F - 5 |\phi''_{2F}|)
\left[ h  \sum_{j=-N+1}^N  (Dv_j)^2 \right]. \qedhere
\end{align*}
\end{proof}

The preceding stability Lemma~\ref{qce_stab} and the discrete
Poincar\'e inequality,
\begin{equation}\label{poincare}
\lpnorm{\vb}{2} \le \frac h{2\sin \frac{\pi h}2}\lpnorm{\Db \vb}{2}
\le \frac 12 \lpnorm{\Db \vb}{2} \quad \text{if }
\sum_{j=-N+1}^N v_j=0
\end{equation}
for $0<h\le 1,$
 give the following stability
result in the $\lpnorm{\cdot}{2}$ norm.  The proof of \eqref{poincare}
follows from verifying that
$(2\sin \frac{\pi h}2)/h$ is the smallest eigenvalue of $D^TD.$
\begin{lemma}
\label{qce_2}
If $\nu := \phi''_F - 5 |\phi''_{2F}|$ and
\begin{equation}\label{2}
L^{qce,h} \vb = \bb,
\end{equation}
where $\sum^N_{j=-N+1} b_j = 0,$
then
\begin{equation}\label{22}
\lpnorm{D \vb}{2} \leq \frac{1}{2\nu} \lpnorm{\bb}{2}.
\end{equation}
\end{lemma}
\begin{proof}
The result~\eqref{22} follows from taking the inner product of~\eqref{2} with
$\vb$ and then using the positive definiteness inequality~\eqref{stab} and
the Poincar\'e inequality \eqref{poincare}.
\end{proof}

\subsection{Quasi-nonlocal Quasicontinuum Approximation}
The quasi-nonlocal quasicontinuum approximation (QNL) is  similar
to the QCE approximation, but it modifies the interactions around the interface
in order to remove $\gb$ from the elastic force.  The quasi-nonlocal
atoms $\pm K, \pm (K+1)$ interact directly with any atoms in the
atomistic region within the next
nearest neighbor cut-off, but interact as in the continuum region with
other all other atoms.   That is, unlike the atomistic model and continuum
approximation, the form of energy contributions for quasi-nonlocal atoms
depends on the type (atomstic, continuum, or quasi-nonlocal) of the
neighboring atoms.  For example, the energy
contribution for $j=K$ is
\begin{equation*}
\begin{split}
\E^{q,h}_K(\ub) :=
&\frac{h}{2} \left[ (\phi'_F + 2 \phi'_{2F})
\left[\frac{u_{K+1} - u_{K}}{h}\right]
+ \half (\phi''_F + 4 \phi''_{2F}) \left[\frac{u_{K+1} - u_{K}}{h}\right]^2 \right]\\
&+\frac{h}{2} \Bigg[ \phi'_F \left[\frac{u_{K} - u_{K-1}}{h}\right]
  + \half \phi''_{F} \left[\frac{u_{K} - u_{K-1}}{h}\right]^2 \\
& \qquad \quad + \phi'_{2F} \left[\frac{u_{K} - u_{K-2}}{h}\right]
  + \half \phi''_{2 F} \left[\frac{u_{K} - u_{K-2}}{h}\right]^2 \Bigg]
\end{split}
\end{equation*}
and the energy contribution for $j=K+1$ is
\begin{equation*}
\begin{split}
\E^{q,h}_{K+1}(\ub) :=
&\frac{h}{2} \left[ (\phi'_F + 2 \phi'_{2F})
\left[\frac{u_{K+2} - u_{K+1}}{h}\right]
+ \half (\phi''_F + 4 \phi''_{2F}) \left[\frac{u_{K+2} - u_{K+1}}{h}\right]^2 \right]\\
&+\frac{h}{2} \Bigg[ \phi'_F \left[\frac{u_{K+1} - u_{K}}{h}\right]
  + \half \phi''_{F} \left[\frac{u_{K+1} - u_{K}}{h}\right]^2 \\
& \qquad \quad + \phi'_{2F} \left[\frac{u_{K+1} - u_{K-1}}{h}\right]
  + \half \phi''_{2 F} \left[\frac{u_{K+1} - u_{K-1}}{h}\right]^2 \Bigg].
\end{split}
\end{equation*}
The QNL energy is then
\begin{equation}
\label{qcnTot}
\begin{split}
\E^{qnl,h}(\ub) :=
 &\sum_{l = -N+1}^{-K-2} h W(Du_l)
+ \half h W(Du_{-K-1}) + \sum_{j=-K-1}^{-K}   \E_{j}^{q,h}\left(\ub\right)
+ \sum_{j=-K+1}^{K-1}   \E_{j}^{a,h}\left(\ub\right) \\
&+ \sum_{j=K}^{K+1}   \E_{j}^{q,h}\left(\ub\right)+ \half h W(Du_{K+2})
+ \sum_{l= K+3}^{N} h W(Du_l).
\end{split}
\end{equation}

The QNL equilibrium equations are
\begin{equation*}
L^{qnl,h} \ub_{qnl} = \fb,
\end{equation*}
where
\begin{equation*}
\label{Lqnl}
\begin{split}
(L^{qnl,h} \ub)_j &= \phi''_F \frac{-u_{j+1} +2 u_j - u_{j-1}}{h^2} \\
&+ \begin{cases}
\displaystyle
 4 \phi''_{2 F} \frac{-u_{j+2} +2 u_j - u_{j-2}}{4 h^2},
& 0 \leq j \leq K-1, \\[6pt]
\displaystyle
 4 \phi''_{2 F} \frac{-u_{j+2} +2 u_j - u_{j-2}}{4 h^2}
- \phi''_{2 F} \frac{-u_{j+2} + 2 u_{j+1} - u_{j}}{h^2},
& j = K, \\[6pt]
\displaystyle
4 \phi''_{2 F} \frac{-u_{j+1} +2 u_j - u_{j-1}}{h^2}
+ \phi''_{2 F} \frac{-u_{j} + 2 u_{j-1} - u_{j-2}}{h^2},
& j = K+1, \\[6pt]
\displaystyle
4 \phi''_{2 F} \frac{-u_{j+1} +2 u_j - u_{j-1}}{h^2}, & K+2 \leq j \leq N.
\end{cases}
\end{split}
\end{equation*}
We note that the QNL energy satisfies the symmetry condition
$\E^{qnl,h}(S\ub)=\E^{qnl,h}(\ub),$ so the QNL operator $L^{qnl,h}$
is defined for $j<0$ by the identity
$SL^{qnl,h}S=L^{qnl,h}.$
While we have successfully removed the ghost force terms $\gb,$ QNL is also
not a consistent approximation of the continuum equations~\eqref{bvp}
at the interfacial atoms, such as $j=K$ and $j=K+1$ above.  We will give
a more detailed analysis of the approximation at the interface in
Section~\ref{sec:qnl}.

Our analysis of the QNL error will utilize the following stability result
for the operator $L^{qnl,h}.$
\begin{lemma}\label{qnl_stab}
If $1\le p\le \infty,$ $\nu:=\phi''_{F} - 4 |\phi''_{2 F}|>0,$ and
\begin{equation*}
L^{qnl,h} \vb = \bb
\end{equation*}
where $\sum_{j=-N+1}^N b_j = 0,$
then
\begin{equation}\label{stabqnl}
\begin{split}
h\vb\cdot L^{qnl,h}\vb &\ge \nu \lpnorm{D\vb}{2}^2, \\
\lpnorm{D \vb}{2}  &\leq \frac{1}{2\nu} \lpnorm{\bb}{2}.
\end{split}
\end{equation}
\end{lemma}
\begin{proof}
The proof of the stability result~\eqref{stabqnl} follows the proof of
the stability results for the QCE approximation in Lemmas \ref{qce_stab}
and \ref{qce_2} with the appropriate
modification.
\end{proof}

\begin{remark}
The basic formulation of the QNL method removes the ghost force terms only for
second-neighbor interactions in the 1D case.  A longer-range matching method is
proposed in~\cite{e06} that removes ghost forces for longer-range interactions
by extending the region near the interface which have special energies.

In 2D and 3D, there are similar restrictions on the interaction length that QNL
corrects.  Additional ghost forces arise when the quasicontinuum energy
is extended to allow coarsening in the continuum region~\cite{e06}.
\end{remark}

\section{Convergence of the Energy-Based Quasicontinuum Solution}
\label{sec:comp}
We now analyze the quasicontinuum error
and obtain estimates for its convergence rate
by splitting the residual into two parts.
One portion contains the low order terms, has support
only near the atomistic to continuum interface, and is
oscillatory.  The remainder is higher order, and its
influence will be bounded using the stability results.
We recall that the QCE solution, $\ub_{qce},$ is an odd, periodic
solution of
\begin{equation}\label{recall}
\begin{split}
L^{qce,h} \ub_{qce} =\gb + \fb,\\
\end{split}
\end{equation}
and the continuum model solution is an odd,
periodic function $u_e(x)$ satisfying
\begin{equation}\label{exact2}
\begin{split}
-(\phi''_F + 4 \phi''_{2F}) u''_e = f,\\
\end{split}
\end{equation}
Let $\ub_e$ denote the vector satisfying $u_j = u_e(x_j).$  We will now
derive estimates for the quasicontinuum error $\eb=\ub_e-\ub_{qce}.$

It follows from the QCE equilibrium equation~\eqref{recall} that
\begin{equation}\label{residual}
L^{qce,h}\eb=L^{qce,h} \ub_{e}-L^{qce,h} \ub_{qce}=L^{qce,h} \ub_{e} - \gb -\fb.
\end{equation}
We split the residual $L^{qce,h}\eb$ as
\begin{equation}\label{split}
L^{qce,h} \eb  := \rhob + \sigb,
\end{equation}
where $\rhob$ contains the three lowest-order residual error terms in the interface,
\begin{equation}\label{rhodef}
\begin{split}
\rhob = \begin{cases}
0, & 0 \leq j \leq K-2, \\[3pt]
\ \ \ \left(\frac{1}{2} \phi'_{2F} + \phi''_{2F} u'_{K+1/2} \right)\frac{1}{h}
- \frac{1}{2} \phi''_{2F} u''_{K+1/2},
+ \frac{7}{24} \phi''_{2F} u'''_{K+1/2} h,
& j = K-1, \\[3pt]
-   \left(\frac{1}{2} \phi'_{2F} + \phi''_{2F} u'_{K+1/2} \right)\frac{1}{h}
+ \frac{1}{2} \phi''_{2F} u''_{K+1/2},
+ \frac{5}{24} \phi''_{2F} u'''_{K+1/2} h,
& j = K, \\[3pt]
-   \left(\frac{1}{2} \phi'_{2F} + \phi''_{2F} u'_{K+1/2} \right)\frac{1}{h}
- \frac{1}{2} \phi''_{2F} u''_{K+1/2},
+ \frac{5}{24} \phi''_{2F} u'''_{K+1/2} h,
& j = K+1, \\[3pt]
\ \ \ \left(\frac{1}{2} \phi'_{2F} + \phi''_{2F} u'_{K+1/2} \right)\frac{1}{h}
+ \frac{1}{2} \phi''_{2F} u''_{K+1/2},
+ \frac{7}{24} \phi''_{2F} u'''_{K+1/2} h,
& j = K+2, \\[3pt]
0, & K+3 \leq j \leq N,
\end{cases}
\end{split}
\end{equation}
and $\rho_j = - \rho_{-j}.$
Although $\rho_j=\text{O}(1/h)$ in the interface
$j=K-1,\dots,K+2,$  we will prove that the effect of $\rhob$ on the
error is small away from the interface because it oscillates and the
lowest order terms cancel
in the sum
\begin{equation}\label{osc}
\Delta \rhob := \sum_{j=K-1}^{K+2} \rho_j=h\phi''_{2F} u'''_{K+1/2}.
\end{equation}
The residual term $\rhob$ represents the inconsistency of the operator
$L^{qce,h}$ as a second-order finite difference approximation of the
differential equation~\eqref{exact2}.  This inconsistency is located
only in the interface because the models themselves are second-order
approximations.

The residual term $\sigb$ accounts for the error in approximating the
continuum model~\eqref{bvp} by a second-order finite difference approximation.
We can estimate the residual $\sigb$
from Taylor's Theorem to obtain
\begin{align}
\lpnorm{{\sigb}}{2}\le Ch^2 \Lpnorm{u_e^{(4)}}{2} \label{second}.
\end{align}

Note that since $\ub_e,\ \fb,$ and $\gb$ are odd, and $\rhob$ was constructed to
be odd, then $\sigb$ is
odd as well.  Therefore, we can split the error $\eb$ as
\begin{equation*}
\eb= \eb_{\rhob} +\eb_{\sigb}
\end{equation*}
such that
\begin{equation}
\label{eqsplit}
\begin{split}
L^{qce,h} \eb_{\rhob}&=\rhob,\qquad  e_{\rho,j}=-e_{\rho,-j},\\
L^{qce,h} \eb_{\sigb}&=\sigb, \qquad  e_{\sigma,j}=-e_{\sigma,-j}.
\end{split}
\end{equation}

\subsection{Global Discretization Error, $\eb_\sigb$}
We now have by the stability~\eqref{stab} of $L^{qce,h}$ and the estimate
of the residual~\eqref{second} that
\begin{equation}\label{sig}
\begin{split}
 \lpnorm{D\eb_{\sigb}}{2} &\le Ch^2 \Lpnorm{u_e^{(4)}}{2}.
\end{split}
\end{equation}
We can extend the bound to
\begin{lemma}
\label{esigbound}
For $\eb_\sigb$ defined in~\eqref{eqsplit}, we have
\begin{equation}\label{siginfty}
\begin{split}
\lpnorm{\eb_{\sigb}}{\infty} &\leq \sqrt{2} \lpnorm{D \eb_{\sigb}}{2}
\le Ch^2 \Lpnorm{u_e^{(4)}}{2},
\end{split}
\end{equation}
\begin{equation}
\label{invass}
\lpnorm{D \eb_{\sigb}}{p} \leq
\begin{cases}
Ch^{2} \Lpnorm{u_e^{(4)}}{2},
& 1 \leq p \leq 2, \\
Ch^{\frac{3}{2} + \frac{1}{p}} \Lpnorm{u_e^{(4)}}{2},
& 2 \leq p \leq \infty.
\end{cases}
\end{equation}
\end{lemma}
\begin{proof}
We obtain the
Poincar\'e inequality~\cite{brenner}
\begin{equation}\label{poincare3}
\lpnorm{\vb}{\infty} \le \lpnorm{\Db \vb}{1} \le  \sqrt{2} \lpnorm{D \vb}{2}
\end{equation}
for all odd periodic $\vb$ from the identity
\begin{equation*}
v_j=
\begin{cases}
\sum_{\ell=1}^j h (Dv_\ell)&\text{if }j>0,\\
-\sum^{0}_{\ell=j-1}h (Dv_\ell)&\text{if }j<0,
\end{cases}
\end{equation*}
which gives the first inequality in~\eqref{poincare3}.
The second follows from H\"older's inequality.  We can then
obtain the error estimate~\eqref{siginfty} for
$\lpnorm{\eb_{\sigb}}{\infty}$
from the
Poincar\'e inequality~\eqref{poincare3} and the
bound~\eqref{sig}.

The ``inverse'' estimate~\cite{brenner}
\begin{equation}\label{inverse}
\lpnorm{ D\vb}{\infty}\le h^{-1/2} \lpnorm{ D\vb}{2}
\end{equation}
for all periodic $\vb,$ and the H{\"o}lder estimates~\cite{rudin}
\begin{equation}\label{holder}
\lpnorm{ D\vb}{p}\le
\begin{cases}
2^{\frac{2-p}{2p}}\lpnorm{ D\vb}{2},\qquad &1 \leq p\leq 2,\\
\lpnorm{ D\vb}{2}^{2/p}\lpnorm{ D\vb}{\infty}^{1-2/p},
\qquad &2 \leq p \leq \infty,
\end{cases}
\end{equation}
combine to prove~\eqref{invass} by taking $\vb=\eb_\sigb.$
\end{proof}

\subsection{Interfacial coupling error, $\eb_{\rhob}$}
In the following, we will bound the error, $\eb_{\rhob},$ by constructing and
estimating an explicit odd solution of
\begin{equation}\label{rhoeq}
L^{qce,h} \eb_{\rhob} = \rhob.
\end{equation}

Since $\rho_j$ is zero for all $j$ except $j = \pm\{K-1,K,K+1,K+2\},$
$\eb_{\rhob}$ satisfies a second-order, homogeneous recurrence relation
in the interior of the continuum region and a fourth-order, homogeneous recurrence relation
in the interior of the atomistic region.  Therefore, $e_{\rho,j}$ is
linear for $j \geq K+3$ or $j \leq -K-3,$ and it is the sum of a
linear solution and exponential solution for $-K+2 \leq j \leq K-2.$
The coefficients for these solutions are determined by the equations
in the atomistic to continuum interface.

The homogeneous atomistic difference scheme
\begin{equation}\label{ato}
- \phi''_{2F}  u_{j+2} - \phi''_F  u_{j+1} + (2 \phi''_F + 2
\phi''_{2F})  u_{j} - \phi''_F  u_{j-1} - \phi''_{2F}  u_{j-2}
= 0
\end{equation}
has characteristic equation
\begin{equation*}
- \phi''_{2F} \Lambda^2 - \phi''_F \Lambda + (2 \phi''_F + 2 \phi''_{2F})
- \phi''_F \Lambda^{-1} - \phi''_{2F} \Lambda^{-2}
= 0,
\end{equation*}
with roots
\begin{equation}\label{roots}
1, 1, \lambda, \frac{1}{\lambda},
\end{equation}
where
\begin{equation*}
\lambda=\frac{(\phi''_F + 2 \phi''_{2F})
    + \sqrt{(\phi''_F)^2 + 4 \phi''_F \phi''_{2F}}}{-2 \phi''_{2F}}.
\end{equation*}
Based on the assumptions on $\phi$ in~\eqref{posdef}and~\eqref{assume}
and
we have that $\lambda > 1.$  We note that if $\phi''_{2F}$ were positive
contrary to
assumption ~\eqref{assume},
then $\lambda$ would
be negative which would give an oscillatory error in the atomistic region.
General solutions of the homogeneous atomistic
equations ~\eqref{ato}
have the form $ u_{j} = C_1 + C_2 h j + C_3 \lambda^j + C_4
\lambda^{-j},$ but seeking an odd solution reduces this to the form
$ u_{j} = C_2 hj + C_3 (\lambda^j - \lambda^{-j}).$

The odd solution of the approximate error equations
is thus of the form
\begin{equation}
\label{form}
 e_{\rho,j} = \begin{cases}
m_1 hj + \beta \left(\frac{\lambda^j - \lambda^{-j}}{\lambda^K}\right),
& 0 \leq j \leq K, \\
m_2 hj - m_2 + \hat{e}_{K+1} , & j = K+1, \\
m_2 hj -m_2, & K+2 \leq j \leq N,
\end{cases}
\end{equation}
where expressing the unknown $ e_{\rho,K+1}$ using a perturbation of the
linear solution, $\hat{e}_{K+1},$ simplifies the solution
of the equilibrium equations.
The four coefficients $m_1,\, m_2,\, \hat{e}_{K+1}, \text{ and } \beta$ can be
found by satisfying the four equilibrium equations in the interface, $j = K-1,\dots,K+2.$
Summing the equilibrium equations across the interface gives
\begin{equation*}
\begin{split}
\Delta \rhob &= \sum_{j=K-1}^{K+2} \rho_j
   = \sum_{j=K-1}^{K+2} (L^{qce,h} \eb_{\rhob})_j \\
  &= \phi''_F \left[ \frac{ e_{\rho,K-1} -  e_{\rho,K-2}}{h^2} \right]
     + 4 \phi''_{2F} \left[ \frac{ e_{\rho,K} +  e_{\rho,K-1} -  e_{\rho,K-2}
                            -  e_{\rho,K-3}}{4 h^2} \right]\\
  &  \qquad - (\phi''_F + 4 \phi''_{2F}) \left[ \frac{ e_{\rho,K+3}
                   -  e_{\rho,K+2}}{h^2} \right]\\
  &= (\phi''_F + 4 \phi''_{2F}) \left(\frac{m_1}{h} - \frac{m_2}{h}\right).
\end{split}
\end{equation*}
The cancellation of the exponential terms in the final equality holds
because
\begin{equation*}
\phi''_{2F}(\lambda^{K}-\lambda^{-K})+(\phi''_F+\phi''_{2F})(\lambda^{K-1}-\lambda^{-K+1}
-\lambda^{K-2}+\lambda^{-K+2})+\phi''_{2F}(-\lambda^{K-3}+\lambda^{-K+3})=0,
\end{equation*}
which can be seen by summing \eqref{ato} with the homogeneous solution
$ y_j=-\lambda^j$ for $j=-K+2,\dots,K-2.$
Thus, we have from summing the equilibrium equations~\eqref{rhoeq} across the
interface that
\begin{equation}
\label{m}
m_1 = m_2 + \frac{h \Delta \rhob}{\phi''_F + 4 \phi''_{2F}}.
\end{equation}
The equality ~\eqref{m} can be interpreted as saying that the
interfacial residual $\rhob$ acts as a source $f=\Delta \rhob$ in the
continuum equations ~\eqref{exact2} at
$x=x_K.$

\begin{lemma}
\label{erhobound}
For $\eb_\rhob$ defined in~\eqref{eqsplit}, we have that
\begin{equation}\label{rho2}
\begin{split}
\lpnorm{\eb_{\rhob}}{\infty} &\leq C h (1+|u'_{K+1/2}| +h|u''_{K+1/2}|
                  +h|u'''_{K+1/2}|), \\
\lpnorm{D \eb_{\rhob}}{p} &\leq C h^{1/p}(1+|u'_{K+1/2}| +h|u''_{K+1/2}|
                  +h|u'''_{K+1/2}|),
\end{split}
\end{equation}
where $C>0$ is independent of $h, K,$ and $p,\ 1 \leq p \leq \infty.$
\end{lemma}

\begin{proof}
We will set up the system of equations for the coefficients in~\eqref{form}
and bound the decay of the coefficients.  We split the
interface equations as  $(A + h B) \xb = \bb,$ where
\begin{equation}
\begin{split}
A &= \left[
\begin{array}{rrrr}
  0
  &\frac{1}{2} \phi''_{2F}
  & -\frac{1}{2} \phi''_{2F}
  &\phi''_{2F} \gamma_{K+1} - \frac{1}{2} \phi''_{2F} \gamma_{K-1} \\[3pt]
  0
  & \phi''_F + \frac{5}{2} \phi''_{2F}
  & -\phi''_F - 2 \phi''_{2F}
  &\phi''_F \gamma_{K+1} + \phi''_{2F} \gamma_{K+2} + \frac{3}{2} \phi''_{2F} \gamma_{K} \\[3pt]
  0
  &-\phi''_F -\frac{5}{2} \phi''_{2F}
  &2\phi''_F +\frac{13}{2}\phi''_{2F}
  &-\phi''_F\gamma_{K} -2\phi''_{2F} \gamma_{K} - \frac{1}{2} \phi''_{2F} \gamma_{K-1} \\[3pt]
  0
  &- \frac{1}{2} \phi''_{2F}
  & -\phi''_F - 4 \phi''_{2F}
  &- \frac{1}{2} \phi''_{2F} \gamma_{K}
\end{array}
\right] ,\\
B &= \left[
\begin{array}{rrrr}
 (\half K + \frac{3}{2}) \phi''_{2F}
&  -(\half K + \half) \phi''_{2F}
&0 & 0\\[3pt]
(K+1) \phi''_F + (\frac{5}{2} K + 2) \phi''_{2F}
&- (K+1) \phi''_F - (\frac{5}{2} K + 3)\phi''_{2F}
 & 0 & 0 \\[3pt]
-K\phi''_F  - (\frac{5}{2} K - \half) \phi''_{2F}
& K \phi''_F  - (\frac{5}{2} K - \frac{3}{2}) \phi''_{2F}
 & 0 & 0 \\[3pt]
-\half K \phi''_{2F}
& (\half K+1)\phi''_{2F}
 & 0 & 0
\end{array}
\right] , \\
\xb &=
\left[
\begin{array}{l}
m_1 \\
m_2 \\
\hat{e}_{K+1} \\
\beta
\end{array}
\right], \qquad
\bb = h^2 \left[
\begin{array}{l}
\rho_{K-1} \\
\rho_{K} \\
\rho_{K+1} \\
\rho_{K+2}
\end{array}
\right].
\end{split}
\end{equation}
Using the equality~\eqref{m},
we rewrite the above as $(\At_K + h \Bt) \xbt = \bbt$ where
\begin{equation}
\begin{split}
\At_K &= \left[
\begin{array}{rrr}
 \frac{1}{2} \phi''_{2F}
  & -\frac{1}{2} \phi''_{2F}
  &\phi''_{2F} \gamma_{K+1} - \frac{1}{2} \phi''_{2F} \gamma_{K-1} \\[3pt]
-\phi''_F -\frac{5}{2} \phi''_{2F}
  &2\phi''_F +\frac{13}{2}\phi''_{2F}
  &-\phi''_F\gamma_{K} -2\phi''_{2F} \gamma_{K}
   - \frac{1}{2} \phi''_{2F} \gamma_{K-1} \\[3pt]
- \frac{1}{2} \phi''_{2F}
  & -\phi''_F - 4 \phi''_{2F}
  &- \frac{1}{2} \phi''_{2F} \gamma_{K}
\end{array}
\right] ,\\
\Bt &= \left[
\begin{array}{rrr}
 \phi''_{2F} & 0 & 0\\
-\phi''_{2F} & 0 & 0 \\
 \phi''_{2F} & 0 & 0
\end{array}
\right] , \qquad
\xbt =
\left[
\begin{array}{l}
m_2 \\
\hat{e}_{K+1} \\
\beta
\end{array}
\right], \\
\bbt &= h^2 \left[
\begin{array}{c}
\rho_{K-1} \\
\rho_{K+1} \\
\rho_{K+2}
\end{array}
\right]
- h^2 \frac{\Delta \rhob}{\phi''_F + 4 \phi''_{2F}} \left[
\begin{array}{c}
( \half K + \frac{3}{2}) \phi''_{2F} \\[3pt]
-K \phi''_F - (\frac{5}{2} K - \half) \phi''_{2F} \\[3pt]
- \half K \phi''_{2F}
\end{array}
\right],
\end{split}
\end{equation}
and $\gamma_j = \frac{\lambda^j - \lambda^{-j}}{\lambda^{K}}.$
We have omitted the second equation, as the full system is linearly
dependent after the elimination of $m_1$ by \eqref{m}.

We note that $\At_K,$ $\Bt,$ and $\bbt$ do not depend on $h$ directly,
though $\At_K$ may have indirect dependence if $K$ scales with $h.$
Therefore, we can neglect $\Bt$ for sufficiently small $h$ provided
that $\At_K^{-1}$ exists and is bounded uniformly in $K.$
The following lemma, proven in~\cite{dobs08d}, gives such a bound for
$\At_K.$

\begin{lemma}
\label{fullrank}
For all $K$ satisfying $2\le K\le N-2,$ the matrix $\At_K$ is nonsingular and
$||\At_K^{-1}|| \leq C$ where $C>0$ is
independent of $K$ and $h.$
\end{lemma}

Due to the definition of $\rhob$ \eqref{rhodef} and $\Delta \rhob$
~\eqref{osc}, we have that
\[
||\bbt||_{\ell^\infty} \leq C (h+h|u'_{K+1/2}| +h^2|u''_{K+1/2}|
                                 +h^2|u'''_{K+1/2}|).
\]
The $|u'''_{K+1/2}|$ contribution from $\Delta \rho$ does not have $h^3$
as coefficient since $K$ may scale linearly with $N=1/h.$
In general, we only have that $h K \leq 1.$

Applying Lemma~\ref{fullrank},
we see that $\xbt$ is O($h$), and by~\eqref{m}, so is $\xb.$
From~\eqref{form}, we finally conclude~\eqref{rho2}.
\end{proof}

\subsection{Total error}

Combining the estimates~\eqref{invass} and \eqref{siginfty} given
in Lemma~\ref{esigbound} for $\eb_{\sigb}$
with the estimate~\eqref{rho2} in Lemma~\ref{erhobound} for $\eb_{\rhob},$
we obtain from the triangle inequality that
\begin{theorem}
\label{qcethm}
Let $\eb$ denote the QCE error.
Then for $1 \leq p \leq \infty,$ $2\le K\le N-2,$ and $h$ sufficiently small,
the error can be bounded by
\begin{equation*}
\begin{split}
\lpnorm{\eb}{\infty} &\leq C h \left(1 + |u'_{K+1/2}| +h|u''_{K+1/2}|
       +h|u'''_{K+1/2}|+h\Lpnorm{u_e^{(4)}}{2}\right), \\
\lpnorm{D \eb}{p} &\leq C h^{1/p} \left(1 + |u'_{K+1/2}| +h|u''_{K+1/2}|
       +h|u'''_{K+1/2}|+h\Lpnorm{u_e^{(4)}}{2}\right).
\end{split}
\end{equation*}
\end{theorem}

We note that the above argument giving optimal order estimates in Theorem~\ref{qcethm} only
utilized the estimate $\Delta \rhob=$O($1$), rather than the optimal estimate
$\Delta \rhob=$O($h$) given in ~\eqref{osc}.

Although our theorems give
optimal order rates of convergence, our assumptions on the required regularity on $u_e$ is
not optimal.  We have assumed for simplicity of exposition that
$\Lpnorm{u_e^{(4)}}{2}<\infty$ and used the estimate \eqref{second}.  Lower order estimates
for $\sigb$ such as
\[
\lpnorm{{\sigb}}{2}\le Ch^{2-s} \Lpnorm{u_e^{(4-s)}}{2},\qquad 2<s<4,
\]
can be used to reduce the regularity assumptions on $u_e$ and still obtain
optimal rates of convergence.  Assuming the full regularity on $u_e$ also makes
possible the precise identification and removal of the lower order terms in the error by the modification of the atomistic-to-continuum coupling scheme.

\section{Convergence of the Quasi-nonlocal Quasicontinuum Solution}
\label{sec:qnl}
For the quasi-nonlocal approximation, we split the residual as
\begin{equation}
\label{elqnl}
\begin{split}
L^{qnl,h} \eb = L^{qnl,h} \ub_e &- \fb = \rhob + \sigb,
\end{split}
\end{equation}
where
\begin{equation}
\begin{split}
\rhob = \begin{cases}
0, & 0 \leq j \leq K-1,\\
- \phi''_{2F} u''_{K+1/2} - \frac{1}{2} \phi''_{2F} u'''_{K+1/2}h, & j = K, \\
\hphantom{-} \phi''_{2F} u''_{K+1/2}
             - \frac{1}{2}  \phi''_{2F} u'''_{K+1/2}h, & j = K+1,\\
0, & K+2 \leq j \leq N,
\end{cases}
\end{split}
\end{equation}
and where
\begin{equation*}
\begin{split}
\lpnorm{\sigb}{p} &\leq C h^2 \Lpnorm{ u_e^{(4)}}{p}.
\end{split}
\end{equation*}

The residual maximum norm $\lpnorm{\rhob}{\infty}$ here is O($1$) as opposed to the energy-based quasicontinuum which
has a O($1/h$) residual maximum norm.
However, the sum of $\rhob$ is similarly O($h$), that is,
\begin{equation}\label{neg2}
\Delta \rhob=-h\phi''_{2F} u'''_{K+1/2}.
\end{equation}

A similar argument as in the QCE case follows.
We split the error as
\begin{equation*}
\eb = \eb_\rhob + \eb_\sigb,
\end{equation*}
where
\begin{equation*}
\begin{split}
L^{qnl,h} \eb_{\rhob}&=\rhob,\qquad  e_{\rho,j}=-e_{\rho,-j},\\
L^{qnl,h} \eb_{\sigb}&=\sigb, \qquad  e_{\sigma,j}=-e_{\sigma,-j}.
\end{split}
\end{equation*}
The same arguments apply to give the bounds~\eqref{invass}
and~\eqref{siginfty} on $\eb_{\sigb}.$  Thus, we need to work through
the modified argument to bound $\eb_\rhob.$
Since $\rhob$ is non-zero only at $j = \pm \{K, K+1\},$
the odd solution $\eb_\rhob$ has the form
\begin{equation}
e_{\rho,j} = \begin{cases}
m_1 hj + \beta (\frac{\lambda^j - \lambda^{-j}}{\lambda^K}), & 0\leq j\leq K,\\
m_2 hj - m_2, & K+1 \leq j \leq N.
\end{cases}
\end{equation}
Summing across the interface again gives
\begin{equation*}
\begin{split}
\Delta \rhob &:= \sum_{j=K-1}^{K+2} \rho_j
   = \sum_{j=K-1}^{K+2} (L^{qce,h} \eb_{\rhob})_j \\
  &= (\phi''_F + 4 \phi''_{2F}) \left(\frac{m_1}{h} - \frac{m_2}{h}\right).
\end{split}
\end{equation*}
Thus, we have again that
\begin{equation}
\label{m2}
m_1 = m_2 + \frac{h \Delta \rhob}{\phi''_F + 4 \phi''_{2F}}.
\end{equation}

We focus on the equations at $j = K-1, K, \text{ and } K+1$ and split the
interface equations as  $(A + h B) \xb = \bb,$ where
\begin{equation}
\begin{split}
A &= \left[
\begin{array}{rrr}
  0
  &\phi''_{2F}
  &\phi''_{2F} \gamma_{K+1}\\[3pt]
  0
  & \phi''_F + 2 \phi''_{2F}
  &\phi''_F \gamma_{K+1} + \phi''_{2F} \gamma_{K+2} + \phi''_{2F} \gamma_{K} \\[3pt]
  0
  &-\phi''_F -3 \phi''_{2F}
  &-\phi''_F \gamma_{K} - 2\phi''_{2F} \gamma_{K} - \phi''_{2F} \gamma_{K-1} \\[3pt]
\end{array}
\right], \\
B &= \left[
\begin{array}{rrr}
(K + 1) \phi''_{2F}
& -(K + 1) \phi''_{2F}
&0 \\[3pt]
(K+1) (\phi''_F + \phi''_{2F})
&-(K+1) (\phi''_F + \phi''_{2F})
 & 0  \\[3pt]
-K (\phi''_F + 3 \phi''_{2F}) + \phi''_{2F}
& K (\phi''_F + 3 \phi''_{2F}) - \phi''_{2F}
 & 0  \\[3pt]
\end{array}
\right] , \\
\xb &=
\left[
\begin{array}{l}
m_1 \\
m_2 \\
\beta
\end{array}
\right], \qquad
\bb = h^2 \left[
\begin{array}{l}
\rho_{K-1} \\
\rho_{K} \\
\rho_{K+1}
\end{array}
\right].
\end{split}
\end{equation}
Using the equality~\eqref{m2},
we rewrite the above as $\At_K  \xbt = \bbt$ where
\begin{equation}
\begin{split}
\At_K &= \left[
\begin{array}{rr}
  \phi''_{2F}
  &\phi''_{2F} \gamma_{K+1}\\[3pt]
   \phi''_F + 2 \phi''_{2F}
  &\phi''_F \gamma_{K+1} + \phi''_{2F} \gamma_{K+2} + \phi''_{2F} \gamma_{K} \\[3pt]
\end{array}
\right], \\
\xb &=
\left[
\begin{array}{l}
m_2 \\
\beta
\end{array}
\right], \qquad
\bbt = h^2 \left[
\begin{array}{l}
\rho_{K-1} \\
\rho_{K} \\
\end{array}
\right]
+ h^2 \frac{\Delta \rhob}{\phi''_F + \phi''_{2F}} \left[
\begin{array}{r}
(K + 1)  \phi''_{2F}  \\[3pt]
(K+1) (\phi''_F + \phi''_{2F}) \\[3pt]
\end{array}
\right].
\end{split}
\end{equation}
We have omitted the second equation, as the full system is linearly
dependent after substitution of $m_1.$  We have that $\At_K$ has full
rank and
$||\bbt||_{\ell^\infty} \leq Ch^2 (|u''_{K+1/2}| + |u'''_{K+1/2}|),$
so that we obtain the following error estimate for the
quasi-nonlocal approximation.

\begin{theorem}
\label{qnlthm}
Let $\eb$ be the solution to the quasi-nonlocal error
equation~\eqref{elqnl}.
Then for $1 \leq p \leq \infty,$ $2\le K\le N-2,$ and $h$ sufficiently small,
the error can be bounded by
\begin{equation*}
\begin{split}
\lpnorm{\eb}{\infty} &\leq C h^2 \left(|u''_{K+1/2}|
      +|u'''_{K+1/2}|+\Lpnorm{u_e^{(4)}}{2}\right), \\
\lpnorm{D \eb}{p} &\leq C h^{1+1/p}\left(|u''_{K+1/2}|
      +|u'''_{K+1/2}|+\Lpnorm{u_e^{(4)}}{2}\right),
\end{split}
\end{equation*}
where $C>0$ is independent of $h, K,$ and $p.$
\end{theorem}

We note that the proof above of the optimal order estimates for the quasi-nonlocal approximation
does use the full O($h$) order of the estimate~\eqref{neg2} for $\Delta\rhob.$

}

\end{document}